\documentclass[11pt,a4paper]{article}

\usepackage[latin1]{inputenc}
\usepackage{amssymb}
\usepackage{url}
\usepackage[T1]{fontenc}
\usepackage{xspace}
\usepackage{color}
\usepackage{psfrag}
\usepackage{mathrsfs}
\usepackage{upgreek} 							
\usepackage{makeidx} 							
\usepackage{fancyhdr}
\usepackage{fancybox}
\usepackage[rm]{titlesec}
\usepackage{float}              	
\usepackage[normalem]{ulem}

\usepackage[english]{babel}
\usepackage{indentfirst, amsfonts, amsmath, amsthm, amscd}
\usepackage{bezier}
\usepackage{biocon}
\usepackage{color}
\usepackage{setspace}
\usepackage{graphicx}
\usepackage{float}
\usepackage{tabularx}
\usepackage{longtable}
\usepackage[all]{xy}
\usepackage[usenames,dvipsnames]{pstricks}
\usepackage{epsfig}
\usepackage{pst-grad}
\usepackage{pst-plot}
\usepackage{setspace}

\newtheorem{theorem}{{Theorem}}[section]

\newtheorem{definition}[theorem]{{Definition}}
\newtheorem{proposition}[theorem]{{Proposition}}

\newcommand{\R}{\mathbb{R}}
\newcommand{\C}{\mathbb{C}}
\newcommand{\N}{\mathbb{N}}

\newcommand{\D}{\mathbb{D}}
\newcommand{\s}{\mathbb{S}}

\newcommand{\dif}{\mathsf{D}}
\newcommand{\J}{\mathsf{J}}

\title{Global phase portraits of a SIS model}
\date{}
\author{\scshape{Regilene D. S. Oliveira} and \scshape{Alex C. Rezende} \\ \small{\it Departamento de Matem\'{a}tica, ICMC, USP, S\~{a}o Carlos, Brazil} \\ \small{{\it E-mails:} \url{regilene@icmc.usp.br}, \url{arezende@icmc.usp.br}}}

\begin{document}

\maketitle

\abstract{In the qualitative theory of ordinary differential equations, we can find many papers whose objective is the classification of all the possible topological phase portraits of a given family of differential system. Most of the studies rely on systems with real parameters and the study consists of outlining their phase portraits by finding out some conditions on the parameters. Here, we studied a susceptible-infected-susceptible (SIS) model described by the
differential system $\dot{x}=-bxy-mx+cy+mk$, $\dot{y}=bxy-(m+c)y$, where $b$, $c$, $k$, $m$ are real parameters with $b \neq 0$, $m \neq 0$ \cite{Brauer}. Such system describes an infectious disease from which infected people recover with immunity against reinfection. The integrability of such system has already been studied by Nucci and Leach \cite{NucciLeach} and Llibre and Valls \cite{LlibreValls}. We found out two different topological classes of phase portraits.}


\section{Introduction} \label{intro}

There is a long time that scientists have been curious about the interaction between portions of a population with particular characteristics, e.g. prey and predator interaction \cite{Boyce}. In 1838, Verhulst \cite{Verhulst} proposed the study of the population growth by means of the so called logistic equation (see also \cite{Boyce}). In contrast, this equation has had other applications, for instance in the study of the spread and the evolution of diseases. The means of how diseases spread has stimulated the interest mainly within researchers of the biological field \cite{Shi}.

As suggested above, the application of mathematical tools in different areas of sciences has been frequent and has had a great impact on the scientific and/or experimental conclusions. For example, Lotka (1925) and Volterra (1926) described independently the dynamics present in biological interactions between two species using ordinary differential equations (ODE). Their model is now known as the Lotka-Volterra equations and they are the most studied equations in the qualitative theory of ODE (see \cite{Boyce,Dana} for more details).

Another approach of the ODE is their application in the study of infectious diseases. According to Levin \cite{Levin}, this study represents one of the oldest and richest areas in math\-e\-mat\-i\-cal biology. Besides, continuous models composed by ODE have formed a large part of the traditional mathematical epidemiology literature, mainly because mathematicians have been attracted by applying the ODE's tools, such as their qualitative theory, to the study of infectious diseases in the attempt of using mathematics to contribute positively to the science field and because the mathematical models become indispensable to inform decision-making.

Particularly, we consider the system of first-order ODE
\begin{equation}
    \begin{array}{ccl}
        \dot{x} & = & -bxy - mx + cy + mk, \\
        \dot{y} & = & bxy - (m+c)y, \\
    \end{array}
    \label{SIS}
\end{equation}
where $x$ and $y$ represent, respectively, the portion of the population that has been susceptible to the infection and those who have already been infected. System \eqref{SIS} is a particular case of the class of classical systems known as susceptible-infected-susceptible (SIS) models, introduced by Kermack and McKendrick \cite{Kermack} and studied by Brauer \cite{Brauer}, who has assumed that recovery from the nonfatal infective disease does not yield immunity. In system \eqref{SIS}, $k$ is the population size (susceptible people plus infected ones), $mk$ is the constant number of births, $m$ is the proportional death rate, $b$ is the infectivity coefficient of the typical Lotka-Volterra interaction term and $c$ is the recovery coefficient. As system \eqref{SIS} is assumed to be nonfatal, the standard term removing dead infected people $-ay$ in \cite{Brauer} is omitted. As usual in the literature, all the critical points of system \eqref{SIS} will henceforth be called (endemic) steady states (e.g. see \cite{Leon3}). In the rest of the paper we will study the phase portraits of the differential system \eqref{SIS} with $bm\neq 0$. Note that if $b=0$, then system \eqref{SIS} becomes linear, and if $m=0$, then system \eqref{SIS} satisfies that $\dot{x} + \dot{y}=0$. These last two cases are trivial and non interesting from a biological point of view.

Much has been studied on SIS models. The great part of the studies present only local stability results for which many mathematical tools are available. In contrast, global studies of these models are very limited due to the lack of applicable theories. The main mathematical tool which has been used for this purpose is the Lyapunov function. The principal limitation is that such functions are not the same for all types of systems (see \cite{Leon1,Leon2} and references there listed for further details).

Besides the stability of SIS models, their integrability has also been studied. For example, Nucci and Leach \cite{NucciLeach} have demonstrated that \eqref{SIS} is integrable using the Painlev\'{e} test. Later, Llibre and Valls \cite{LlibreValls} have proved that system \eqref{SIS} is Darboux integrable, and they have shown the explicit expression of its first integral and all its invariant algebraic curves.

Alternatively, the attempt of outlining the global phase portraits of ODE systems is a possible way to determine their global behavior.

Here, the purpose was to classify all the topological classes of the global phase portraits of system \eqref{SIS} using some information in \cite{LlibreValls}.

The main result in this paper is the following:

\begin{theorem} \label{mainthm}
The phase portrait on the Poincar\'{e} disc of system \eqref{SIS} is topologically e\-quiva\-lent to one of the two phase portraits shown in Figure \ref{phaseportraits}, modulo reversibility.
\end{theorem}

\begin{figure}[hbt!]
    \centering
    \includegraphics[scale=0.5]{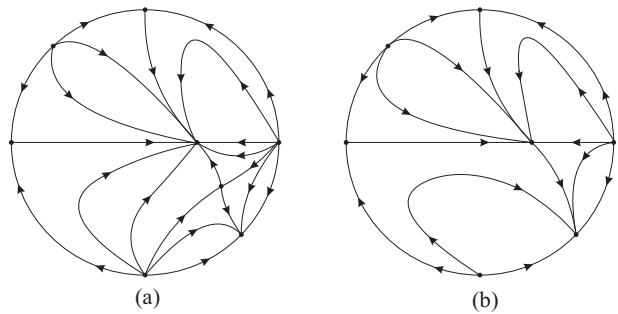}
    \caption{Phase portraits in the Poincar\'{e} disc of system \eqref{SIS}}
    \label{phaseportraits}
\end{figure}

The paper is organized as follows. Basic definitions and stated results about singular points on the plane and their classification are given in Section \ref{secbasicresults}. The reader who is familiar to these concepts can skip this section with no worries of missing new information. In Section \ref{secanalysis}, the local analysis of the finite and infinite singular points of system \eqref{SIS} is done, and so is the existence of at least three invariant straight lines for such system. Section \ref{secmainresult} collects all the information provided in the previous sections and proves Theorem \ref{mainthm}. Finally, Section \ref{secconcl} concludes and discusses the results from the biological point of view.

\section{Basic results} \label{secbasicresults}

In this section, we introduce some definitions and stated theorems which shall be used in the next section for the analysis of the local phase portraits of finite and infinite singularities of system \eqref{SIS}.

\subsection{Singular points}

We denote by $X(x,y) = (P(x,y),Q(x,y))$ a vector field in $\R^2$, associated to the system
\begin{equation}
    \begin{array}{ccl}
        \dot{x} & = & P(x,y), \\
        \dot{y} & = & Q(x,y), \\
    \end{array}
    \label{vectorfield}
\end{equation}
where $P$ and $Q$ are real polynomials in the variables $x$ and $y$ of degree at most $n \in \N$. We recall that the degree $n$ is defined by the maximum of $\deg(P)$ and $\deg(Q)$. We denote by $P_n(\R^2)$ the set of all polynomial vector fields on $\R^2$ of the form \eqref{vectorfield}.

\begin{definition}
A point $p \in \R^2$ is said to be a {\bf singular point}, or a {\bf singularity}, of the vector field $X=(P,Q)$, if $P(p)=Q(p) = 0$.
\end{definition}

We denote by $P_x$ the partial derivative of $P$ with respect to $x$. Suppose that $p \in \R^2$ is a singular point of $X$ and that $P$ and $Q$ are analytic functions in a neighborhood of $p$. Let $\delta = P_x(p) Q_y(p) - P_y(p) Q_x(p)$ and $\tau = P_x(p) + Q_y(p)$.

\begin{definition}
The singular point $p$ is said to be {\bf non-degenerate} if $\delta \neq 0$.
\end{definition}

As a consequence, $p$ is an isolated singular point. Furthermore, $p$ is a {\it saddle} if $\delta < 0$, a {\it node} if $\tau^2 - 4\delta > 0$ ({\it stable} if $\tau <0$, {\it unstable} if $\tau > 0$), a {\it focus} if $\tau^2 - 4\delta < 0$ ({\it stable} if $\tau <0$, {\it unstable} if $\tau > 0$), and either a {\it weak focus} or a {\it center} if $\tau = 0$ and $\delta>0$ \cite{Llibre}.

\begin{definition}
The singular point $p$ is said to be {\bf hyperbolic} if the two eigenvalues of the Jacobian matrix $\J X(p)$ have nonzero real part.
\end{definition}

Hence, the hyperbolic singularities are the non-degenerate ones, except the weak foci and the centers.

\begin{definition}
The degenerate singular point $p$ (i.e. $\delta = 0$), with $\tau \neq 0$, is called {\bf semi-hyperbolic}.
\end{definition}

Again $p$ is isolated in the set of all singular points. Specifically, the next theorem summarizes the results on semi-hyperbolic singularities that shall be used later. For more details, see Theorem 2.19 of \cite{Llibre}.

\begin{theorem} \label{th2.19}
Let $(0,0)$ be an isolated singular point of the vector field $X$ given by
    \begin{equation}
        \begin{array}{ccl}
            \dot{x} & = & A(x,y), \\
            \dot{y} & = & \lambda y + B(x,y), \\
        \end{array}
        \label{eqth2.19}
    \end{equation}
where $A$ and $B$ are analytic in a neighborhood of the origin starting with at least degree 2 in the variables $x$ and $y$. Let $y=f(x)$ be the solution of the equation $\lambda y + B(x,y) = 0$ in a neighborhood of the point $(0,0)$, and suppose that the function $g(x) = A(x,f(x))$ has the expression $g(x) = a x^\alpha + o(x^\alpha)$, where $\alpha \geq 2$ and $a \neq 0$. So, when $\alpha$ is odd, then $(0,0)$ is either an unstable node, or a saddle, depending if $a>0$, or $a<0$, respectively. In the case of the saddle, the separatrices are tangent to the $x$-axis. If $\alpha$ is even, the $(0,0)$ is a saddle-node, i.e. the singular point is formed by the union of two hyperbolic sectors with one parabolic sector. The stable separatrix is tangent to the positive (respectively, negative) $x$-axis at $(0,0)$ according to $a<0$ (respectively, $a>0$). The two unstable separatrices are tangent to the $y$-axis at $(0,0)$.
\end{theorem}

\begin{definition}
The singular points which are non-degenerate or semi-hy\-per\-bol\-ic are called {\bf elementary}.
\end{definition}

Finally, we present two other concepts that shall be used in the next sections.

\begin{definition} \label{defFI}
A function $H: U \subset \R^2 \to \R$ of class $C^1$ is called a {\bf first integral} of system \eqref{vectorfield} if $U$ is an open and dense subset of $\R^2$ and $H$ is constant over the solutions of system \eqref{vectorfield} contained in $U$, i.e. $$P \frac{\partial H}{\partial x} + Q \frac{\partial H}{\partial y} = 0,$$ on the points of $U$. If such $H$ exists, then we say system \eqref{vectorfield} is {\bf integrable}.
\end{definition}

\begin{definition} \label{defIC}
Let $\C[x,y]$ denotes the ring of complex polynomials in the variables $x$ and $y$. The polynomial $f(x,y) \in \C[x,y] \setminus \C$ is {\bf invariant} under system \eqref{vectorfield} if $$P \frac{\partial f}{\partial x} + Q \frac{\partial f}{\partial y} = Kf,$$ for some $K \in \C[x,y]$ called the {\bf cofactor} of $f$.

We say the algebraic curve $f(x,y)=0$ is invariant, if $f(x,y)$ is invariant under system \eqref{vectorfield}.
\end{definition}

\subsection{Poincar\'{e} compactification}

Let $X \in P_n(\R^2)$ be a planar polynomial vector field of degree $n$. The {\it Poincar\'{e} compactified vector field $\pi(X)$ corresponding to $X$} is an analytic vector field induced on $\s^2$ as follows (for more details, see \cite{Llibre}).

Let $\s^2 = \{y=(y_1,y_2,y_3) \in \R^3; \ y_1^2 + y_2^2 + y_3^2 = 1\}$ and $T_y\s^2$ be the tangent plane to $\s^2$ at point $y$. Identify $\R^2$ with $T_{(0,0,1)}\s^2$ and consider the central projection $f: T_{(0,0,1)}\s^2 \to \s^2$. This map defines two copies of $X$ on $\s^2$, one in the northern hemisphere and the other in the southern hemisphere. Denote by $X'$ the vector field $\dif f \circ X$ defined on $\s^2$ except on its equador $\s^1 = \{y \in \s^2; \ y_3 = 0\}$. Then, $\s^1$ is identified to the infinity of $\R^2$.

In order to extend $X'$ to a vector field on $\s^2$, including $\s^1$, $X$ should satisfy suitable conditions. In the case that $X \in P_n(\R^2)$, $\pi(X)$ is the only analytic extension of $y_3^{n-1}X'$ to $\s^2$. On $\s^2 \setminus \s^1$ there exist two symmetric copies of $X$, and knowing the behavior of $\pi(X)$ around $\s^1$, one can conclude the behavior of $X$ in a neighborhood of the infinity. The Poincar\'{e} compactification has the property that $\s^1$ is invariant under the flow of $\pi(X)$. The projection of the closed northern hemisphere of $\s^2$ on $y_3=0$ under $(y_1,y_2,y_3) \mapsto (y_1,y_2)$ is called the {\it Poincar\'{e} disc}, and it is denoted by $\D^2$.

\begin{definition}
Two polynomial vector fields $X$ and $Y$ on $\R^2$ are {\bf topologically equivalent} if there exists a homeomorphism on $\s^2$ preserving the infinity $\s^1$ carrying orbits of the flow induced by $\pi(X)$ into orbits of the flow induced by $\pi(Y)$.
\end{definition}

As $\s^2$ is a differentiable manifold, for computing the expression for $\pi(X)$, consider six local charts $U_i = \{y \in \s^2; \ y_i >0\}$ and $V_i = \{y \in \s^2; \ y_i <0\}$, where $i=1,2,3$, and the diffeomorphisms $F_i: U_i \to \R^2$ and $G_i: V_i \to \R^2$, for $i=1,2,3$, which are the inverses of the central projections from the tangent planes at the points $(1,0,0)$, $(-1,0,0)$, $(0,1,0)$, $(0,-1,0)$, $(0,0,1)$ and $(0,0,-1)$, respectively. Denote by $z=(u,v)$ the value of $F_i(y)$ and $G_i(y)$, for any $i=1,2,3$ (so $z$ represents different things according to the local charts under consideration). Hence, after some easy computations, $\pi(X)$ has the following expressions:
    \begin{equation}
        v^n \Delta(z) \left(Q\left(\frac{1}{v},\frac{u}{v}\right)-u P\left(\frac{1}{v},\frac{u}{v}\right), -v P\left(\frac{1}{v},\frac{u}{v}\right)\right) \ \hbox{in} \ U_1,
    \end{equation}
    \begin{equation}
        v^n \Delta(z) \left(P\left(\frac{u}{v},\frac{1}{v}\right)-u Q\left(\frac{u}{v},\frac{1}{v}\right), -v Q\left(\frac{u}{v},\frac{1}{v}\right)\right) \ \hbox{in} \ U_2,
    \end{equation}
    \begin{equation}
        \Delta(z) (P(u,v),Q(u,v)) \ \hbox{in} \ U_3,
    \end{equation}
where $\Delta(z) = (u^2+v^2+1)^{-(n-1)/2}$. The expression for $V_i$ is the same as that for $U_i$ except for a multiplicative factor $(-1)^{n-1}$. In these coordinates for $i=1,2$, $v=0$ always denotes the points of $\s^1$.

After finding the infinite singular points, it is necessary to classify them. Among the hyperbolic singular points at infinity only nodes and saddles can appear. All the semi-hyperbolic singular points can appear at infinity.

If one of these hyperbolic or semi-hyperbolic singularities at infinity is a (topological) saddle, then the straight line $\{v=0\}$, representing the equator of $\s^2$, is necessarily a stable or unstable manifold, or a center manifold (see Figure \ref{infsing}).

The same property also holds for semi-hyperbolic singularities of saddle-node type. They can hence have their hyperbolic sectors split in two different ways depending on the Jacobian matrix of the system in the charts $U_1$ or $U_2$. The Jacobian matrix can be either
    \begin{equation*}
            \left(\begin{array}{cc}
                \lambda & \star \\ 0 & 0 \\
            \end{array} \right), \ \ \ \hbox{or} \ \ \  \left(\begin{array}{cc}
                                                                  0 & \star \\ 0 & \lambda \\
                                                              \end{array} \right),
    \end{equation*}
with $\lambda \neq 0$. In the first case we say that the {\it saddle-node is of type SN1} and in the second case of {\it type SN2}. The two cases are represented in Figure \ref{infsing}. The sense of the orbits can also be the opposite.
								
\begin{figure}[htb!]
		\centering
		\includegraphics[scale=0.5]{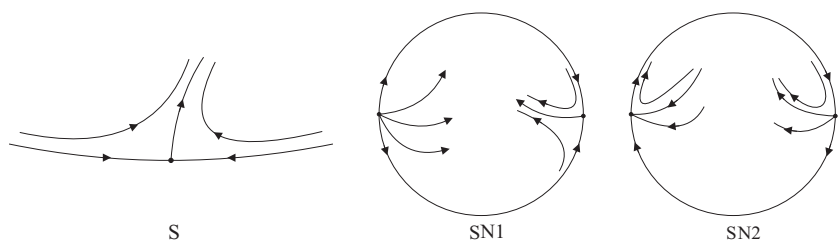}%
		\caption{(S) represents a hyperbolic or semi-hyperbolic saddle on the equator of $\s^2$; (SN1) and (SN2) represent, respectively, saddle-nodes of types SN1 and SN2 of $\pi(X)$ in the equator of $\s^2$}%
		\label{infsing}%
\end{figure}

\section{Analysis of the system} \label{secanalysis}

Here, we provide a mathematical analysis of system \eqref{SIS}.

Provided that $b \neq 0$, system \eqref{SIS} has only two finite singular points: \linebreak $p = ((c+m)/b,(-c+bk-m)/b)$, usually known as endemic steady state, and $q = (k,0)$, usually known as disease-free steady state. In addition, both finite singular points $p$ and $q$ are the same if $bk = c$.

The names given above to each steady states are not an accident. In $q$, the number of susceptible individuals is equal to the population size $k$, whereas the number of infected people is null. On the other hand, the number of susceptible people in $p$ is the recovery coefficient plus the death rate divided by the infection coefficient, while the infected ones are the rest of the population, which leads to the presence of infected people, since $bm \neq 0$. Finally, note that only non-negative values of $x$ and $y$ are interesting here, because they represent the number of individuals.

First, we start with the analysis of the endemic steady state $p$. Translating the singular point $p$ to the origin in system \eqref{SIS}, we obtain
    \begin{equation}
        \begin{array}{ccl}
            \dot{x} & = & -bkx+cx-my-bxy, \\
            \dot{y} & = & (-c+bk-m)x+bxy, \\
        \end{array}
        \label{SIS2}
    \end{equation}
which is equivalent to \eqref{SIS}. The Jacobian matrix of \eqref{SIS2} is given by
    \begin{equation*}
        \J(x,y) = \left(
                        \begin{array}{cc}
                            c-bk-by & -m-bx \\
                            -c+bk-m+by & bx \\
                        \end{array}
                  \right),
    \end{equation*}
which implies that $$\delta = \delta(0,0) = (bk-c-m)m \ \ \hbox{and} \ \ \tau = \tau(0,0) = -bk+c.$$

If $(bk-c-m)m < 0$, then $p$ is a saddle point. On the other hand, if $(bk-c-m)m > 0$, then $p$ is a node point, because $\tau^2 - 4\delta = (c-bk+2m)^2 \geq 0$.

In the case that $(bk-c-m)m = 0$, or equivalently, $m=bk-c$, then $p$ is degenerate. Indeed, it is the case that both finite singular points are the same, i.e. $p=q=(k,0)$. Here, system \eqref{SIS2} becomes
    \begin{equation}
        \begin{array}{ccl}
            \dot{x} & = & -mx-my-bxy, \\
            \dot{y} & = & bxy, \\
        \end{array}
        \label{SIS3}
    \end{equation}
whose Jacobian matrix at $(0,0)$ is
    \begin{equation*}
        \J(0,0) = \left(
                        \begin{array}{cc}
                            -m & -m \\
                            0 & 0 \\
                        \end{array}
                  \right),
    \end{equation*}
so that $p$ is a semi-hyperbolic point. By a linear change of coordinates, system \eqref{SIS3} can be put on the form of system \eqref{eqth2.19} and, applying Theorem \ref{th2.19}, we conclude that $p$ is a saddle-node point.

From the biological point of view, when the death rate $m$ is equal to the portion of the population which becomes infected ($bk$) minus the recovery coefficient, the dynamics around the steady states $p$ and $q$ changes and they become only one point which attracts (the node part) and repels (the saddle part) the orbits in its neighborhood.

Now, we analyze the disease-free steady state $q$. Translating the singular point $q$ to the origin in system \eqref{SIS}, we obtain
    \begin{equation}
        \begin{array}{ccl}
            \dot{x} & = & -mx+cy-bky-bxy, \\
            \dot{y} & = & (-c+bk-m)y+bxy, \\
        \end{array}
        \label{SIS4}
    \end{equation}
which is equivalent to \eqref{SIS}. The Jacobian matrix of \eqref{SIS4} is given by
    \begin{equation*}
        \J(x,y) = \left(
                        \begin{array}{cc}
                            -m-by & c-b(k+x) \\
                            by & -c-m+b(k+x) \\
                        \end{array}
                  \right),
    \end{equation*}
which implies that $$\delta = \delta(0,0) = -(bk-c-m)m \ \ \hbox{and} \ \ \tau= \tau(0,0) = bk-c-2m.$$

If $-(bk-c-m)m < 0$, then $q$ is a saddle point. In contrast, if \linebreak $-(bk-c-m)m > 0$, then $q$ is a node point, because $\tau^2 - 4\delta = (c-bk)^2 \geq 0$.

The case $(bk-c-m)m = 0$, or equivalently $m=bk-c$, has already been studied, and $p=q$ is a semi-hyperbolic saddle-node point.

Finally, we have proved the following:

\begin{proposition} \label{propfinite}
Consider system \eqref{SIS} with $bm\neq 0$ and its two finite steady states $p$ and $q$. Then:
    \begin{enumerate}
        \item If either $m>0$ and $m>bk-c$, or $m<0$ and $m<bk-c$, then $p$ is a saddle and $q$ is a node;
        \item If either $m>0$ and $m<bk-c$, or $m<0$ and $m>bk-c$, then $p$ is a node and $q$ is a saddle;
        \item If $m=bk-c$, then $p=q$ is a semi-hyperbolic saddle-node.
    \end{enumerate}
\end{proposition}

Having classified all the finite singular points, we apply the Poincar\'{e} compactification to study the infinite singularities.

In the local chart $U_1$, where $x= 1/v$ and $y=u/v$, we have:
    \begin{equation}
        \begin{array}{ccl}
            \dot{u} & = & u(b+bu-cv-cuv-kmv^2), \\
            \dot{v} & = & v(bu+mv-cuv-kmv^2), \\
        \end{array}
        \label{SIS5}
    \end{equation}
whose infinite singular points are $(0,0)$ and $(-1,0)$, which are a saddle-node of type SN1 (by Theorem \ref{th2.19}) and a node, respectively.

In the local chart $U_2$, where $x=u/v$ and $y=1/v$, the system
    \begin{equation}
        \begin{array}{ccl}
            \dot{u} & = & -bu-bu^2+cv+cuv+kmv^2, \\
            \dot{v} & = & v(-bu+cv+mv) \\
        \end{array}
        \label{SIS6}
    \end{equation}
has two infinite singular points $(0,0)$ and $(-1,0)$. The latter one is a node and is the same as $(-1,0) \in U_1$, while the former one is a saddle-node of type SN1.

We have just proved the following:

\begin{proposition} \label{propinfinite}
The infinite singular points of system \eqref{SIS} are the origin of charts $U_1$, $V_1$, $U_2$ and $V_2$, which are saddle-node points of type SN1, and $(-1,0)$, belonging to each of the charts $U_1$ and $U_2$, which is a node point.
\end{proposition}

Knowing the local behavior around finite and infinite singular points, another useful tool to describe the phase portraits of differential systems is the existence of invariant curves. The next result shows system \eqref{SIS} has at least three invariant straight lines.

\begin{proposition} \label{propinvlines}
Let $bm \neq 0$. System \eqref{SIS} has at least three invariant straight lines given by $f_1(x,y) = y$ and $f_2(x,y) = k-x-y$, and additionally $f_3(x,y) = k-x$, if $c=bk$.
\end{proposition}

\begin{proof}
By Definition \ref{defIC}, we can find $K_1(x,y) = bx-m-c$, $K_2(x,y) = -m$ and $K_3(x,y) = -m-by$ as the cofactors of $f_1(x,y)$, $f_2(x,y)$ and $f_3(x,y)$ (if $c=bk$), respectively.
\end{proof}

\section{Main result} \label{secmainresult}

From Propositions \ref{propfinite} and \ref{propinfinite} we get all the information about the local behavior of finite and infinite singular points, respectively. Using the continuity of solutions and primary definitions and results of ODE (e.g. $\omega$-limit sets, existence and uniqueness of solutions, the Flow Box Theorem etc. \cite{Llibre}) and the existence of invariant straight lines of system \eqref{SIS} stated by Proposition \ref{propinvlines}, its global phase portraits can be easily drawn.

Essentially, we have only two cases. The finite steady state $q$ is the intersection of the invariant curves $f_1(x,y)=f_2(x,y)=0$, and the other finite steady state $p$ lies on the curve $f_2(x,y)=0$.

According to items (1) and (2) of Proposition \ref{propfinite}, $p$ (respectively, $q$) is a saddle (respectively, a node) the one way and the other a node (respectively, a saddle). The subtle difference here is the position of point $p$. While $q$ remains on the line $\{y=0\}$, $p$ is in the lower part of the Poincar\'{e} disc the one way and the other in the upper part. It is worth mentioning that the four infinite singular points continue to be the same points no matter what conditions are being considered. The phase portrait of both cases above is topologically equivalent to the one which is shown in Figure \ref{phaseportraits}(a).

Item (3) of Proposition \ref{propfinite} assumes the existence of only one finite singular point, $p=q$. Here, when $m=bk-c$, both $p$ and $q$ become only one degenerate singularity which bifurcates into a saddle-node point. Again, no changes are applied to the infinite singular points. The phase portrait of this case is topologically equivalent to the one which is shown in Figure \ref{phaseportraits}(b).

Finally, Theorem \ref{mainthm} has been proved.

\section{Conclusions and discussions} \label{secconcl}

The last section proves the existence of only two classes of global phase portraits of the quadratic system \eqref{SIS}. In the qualitative theory of ODE it is quite important to know the global behavior of solutions of systems and, in general, this is not an easy task. The most frequently used tools for this propose are the study of local behavior along with the (local and global stability), integrability, and also the global phase portrait, which was employed in the present study.

In the case represented by Figure \ref{phaseportraits}(a), it is clear that while the steady state $q$ characterizes the presence of only susceptible individuals, $p$ indicates the mutual presence of susceptible and infected people. Besides, as $q$ is an asymptotically stable node, the disease seems to be controlled and the whole population tends to be healthy but susceptible to be infected again. As $p$ is an unstable saddle steady state, it suggests that there is no harmony between the number of susceptible people and infected ones, although some of the solutions tend to $q$, indicating the control of the disease.

In case of Figure \ref{phaseportraits}(b), all the solutions tend to $q$ (regarding that $x,y>0$), i.e. if $m=bk-c$, the disease is supposed to be controlled and the whole population is inclined to be healthy but susceptible to the reinfection.

\section*{Acknowledgements}

\small{
We would like to thank professor Jaume Llibre, from Universitat Autònoma de Barcelona, for proposing this study, which has contributed for the second author's Ph.D. theses. Both authors are supported by the joint project CAPES/DGU grant 222/2010. The second author has been supported by a Ph.D. CAPES grant.
}

\end{document}